\documentclass[11pt]{article}
\usepackage{amsmath, amscd, amssymb, latexsym, epsfig, color, amsthm,enumerate, mathtools}
\usepackage[all]{xy}
\usepackage{verbatim}
\usepackage{esvect}
\numberwithin{equation}{section}

\newtheorem{theorem}{Theorem}[section]

\newtheorem{corollary}[theorem]{Corollary}
\newtheorem{conjecture}[theorem]{Conjecture}
\newtheorem{lemma}[theorem]{Lemma}

\theoremstyle{definition}
\newtheorem{definition}[theorem]{Definition}

\newcommand{\emb}[1]{\mathbf{#1}}

\DeclareMathOperator{\skel}{Skel}
\DeclareMathOperator\lk{\mathrm{lk}}

\DeclareMathOperator\st{\mathrm{st}}
\DeclareMathOperator{\intr}{\mathrm{int}}

\DeclareMathOperator{\conv}{\mathrm{conv}}
\DeclareMathOperator{\Aff}{\mathrm{aff}}

\DeclareMathOperator{\supp}{\mathrm{supp}}

\DeclareMathOperator{\sign}{\mathrm{sign}}

\newcommand{\R}{{\mathbb R}}

\newcommand{\M}{{\mathcal M}}
\newcommand{\V}{{\mathcal V}}

\newcommand{\Stress}{\mathcal S}

\title{Reconstructing simplicial polytopes from their graphs and affine $2$-stresses}
\author{
	Isabella Novik\thanks{Research of IN is partially\textsl{} supported by NSF grants DMS-1664865 and DMS-1953815, and by Robert R.~\&  Elaine F.~Phelps Professorship in Mathematics. }\\
	\small Department of Mathematics\\[-0.8ex]
	\small University of Washington\\[-0.8ex]
	\small Seattle, WA 98195-4350, USA\\[-0.8ex]
	\small \texttt{novik@uw.edu}
	\and 
	Hailun Zheng\thanks{Research of HZ is partially supported by a postdoctoral fellowship from ERC grant 716424 - CASe.}\\
	\small Department of Mathematical Sciences\\[-0.8ex]
	\small University of Copenhagen\\[-0.8ex]
	\small Universitesparken 5, 2100 Copenhagen, Denmark \\[-0.8ex]
	\small \texttt{hz@math.ku.dk}
}
\begin{document}
	\maketitle
	\begin{abstract} A conjecture of Kalai from 1994 posits that for an arbitrary $2\leq k\leq \lfloor d/2 \rfloor$, the combinatorial type of a simplicial $d$-polytope $P$ is uniquely determined  by the $(k-1)$-skeleton of $P$ (given as an abstract simplicial complex) together with the space of affine $k$-stresses on $P$. We establish the first non-trivial case of this conjecture, namely, the case of $k=2$. We also prove that for a general $k$, Kalai's conjecture holds for the class of $k$-neighborly polytopes. 
	\end{abstract}
	
\section{Introduction}
What partial information about a convex $d$-polytope $P$ is enough to uniquely determine the combinatorial type of $P$? For general polytopes, a result of Gr\"unbaum \cite[Chapter 12]{Gru-book} shows that to reconstruct the face lattice of $P$ we need to know the $(d-2)$-skeleton of $P$. At the same time, for certain classes of polytopes, knowing the graph alone already suffices to determine the combinatorial type. Examples include the class of all simple polytopes (see \cite{Blind-Mani} and \cite{Kalai-simple}) as well as the class of all zonotopes \cite{BjornerEdelmanZieger}.

In the class of simplicial polytopes, neither the graph nor even the $(\lfloor d/2\rfloor-1)$-skeleton provides enough information. Indeed, while any two $\lfloor d/2\rfloor$-neighborly $d$-polytopes on $n$ vertices have isomorphic $(\lfloor d/2\rfloor-1)$-skeleta, there are  $2^{\Theta(n\log n)}$ distinct combinatorial types of such polytopes  \cite{Padrol-13, Shemer}. However, a result of Perles (unpublished) and Dancis \cite{Dancis} asserts that the $\lfloor d/2\rfloor$-skeleton of a simplicial $d$-polytope $P$ does determine the entire face lattice of $P$.  
Another piece of information that allows one to reconstruct the face lattice of $P$ is the space of affine dependencies of vertices of $P$. This observation is at the heart of the theory of Gale diagrams developed by Perles \cite[Chapter 6]{Ziegler}. 

To relate these two very different types of partial information to each other, it is worth pointing out that the affine dependencies of vertices of $P$ are precisely the affine $1$-stresses on $P$, while the space of affine $(\lfloor d/2\rfloor+1)$-stresses of a simplicial $d$-polytope $P$ is trivial. (This latter fact is a consequence of the celebrated $g$-theorem of Stanley \cite{Stanley80} and McMullen \cite{McMullen96} or more precisely of the Strong Lefschetz property of $P$ stated in the language of stresses \cite{Lee94, Lee96}.) In other words, the two reconstruction results mentioned in the previous paragraph are respectively the $k=\lfloor d/2\rfloor+1$ and the $k=1$ cases of the following conjecture of Kalai \cite[Conjecture 7]{Kalai-survey} that can be regarded as a natural conjectural extension of the basic property of Gale diagrams.

\begin{conjecture}  \label{conj:Kalai}
	Let $P$ be a simplicial $d$-polytope and let $1\leq k \leq \lfloor d/2\rfloor+1$. Then the $(k-1)$-skeleton of $P$ (given as an abstract simplicial complex) and the space of (squarefree parts of) affine $k$-stresses of $P$ uniquely determine the combinatorial type of $P$.
\end{conjecture}

All other cases of Conjecture \ref{conj:Kalai} are open at present. The goal of this note is to verify the case  of $k=2$ of this conjecture as well as to prove it for the class of $k$-neighborly polytopes for an arbitrary $k$. In fact, in the complete analogy with the $k=1$ case, we prove that in these two cases, to reconstruct the combinatorial type of $P$, it is enough to know the $(k-1)$-skeleton of $P$ and the set of {\em sign vectors} of affine $k$-stresses on $P$.
In the case of $k=2$, the main ingredients of our proof are basic facts about affine $2$-stresses, such as the cone and gluing lemmas (see \cite{AsimowRothII, Lee-notes}); Whiteley's result asserting that for $d\geq 3$, all simplicial $d$-polytopes are infinitesimally rigid in $\R^d$ \cite{Whiteley-84}; a simple extension of Dehn's lemma that might be of interest on its own; and Balinski's theorem \cite{Balinski}, \cite[Section 3.5]{Ziegler}. Along the way, we state several other conjectures about the set of affine $k$-stresses that, if true, would imply Conjecture \ref{conj:Kalai} in full generality. 

The structure of the rest of this note is as follows. In Section 2 we review several basic definitions related to polytopes and simplicial complexes  as well as introduce some notation. In Section 3, which is also mostly a background section, we discuss a few important results on affine stresses and infinitesimal rigidity. In Section~4, we propose three conjectures of increasing strength all of which imply Conjecture \ref{conj:Kalai}. Then, in Section 5, we prove Conjecture \ref{conj:Kalai} in the two cases described above by verifying one of these stronger conjectures, see Theorems \ref{thm:k=2} and \ref{thm:k-neighb}.

\section{Basics on polytopes and simplicial complexes} 

In this section we collect some basic definitions and results pertaining to simplicial polytopes and simplicial complexes. 

A {\em polytope} $P \subseteq \R^d$ is the convex hull of a finite set of points in $\R^d$. Each hyperplane $H$ in $\R^d$ determines two closed half-spaces of $\R^d$, which we usually denote by $H^+$ and $H^-$. We say that $H$ is a {\em supporting hyperplane} of $P$ if $P$ is contained in one of these two half-spaces and $H\cap P$ is nonempty; the intersection $F= H\cap P$ is then called a {\em proper} face of $P$. We sometimes also refer to the empty set and $P$ itself as {\em non-proper} faces of $P$. The {\em dimension} of a face $F$ is the dimension of the affine span of $F$. In particular, $P$ is a {\em $d$-polytope} if $\dim P=d$. The faces of a $d$-polytope $P$ of dimensions $0, 1$ and $d-1$ are called {\em vertices, edges}, and {\em facets}, respectively. 
By passing from $\R^d$ to the affine span of $P$, we can always assume that $P \subseteq \R^d$ is a full-dimensional polytope. 

An important example of a polytope is a {\em {geometric} $d$-simplex}. It is defined as the convex hull of a set of $d+1$ affinely independent points. Another important family of polytopes is that of $k$-neighborly polytopes: a polytope $P$ is {\em $k$-neighborly} if every $k$-subset of the vertex set of $P$ spans a face of $P$.

An (abstract) {\em simplicial complex} $\Delta$ with vertex set $V=V(\Delta)$ is a {non-empty} collection of subsets of $V$ that is closed under inclusion. The elements of $\Delta$ are called {\em faces} of $\Delta$. A face $F$ of $\Delta$ is an {\em $i$-face} or a face of {\em dimension $i$} if $|F|=i+1$. The {\em $i$-skeleton} of $\Delta$, $\skel_i(\Delta)$, is the set of all faces of $\Delta$ of dimension at most $i$. We often refer to the $1$-skeleton of $\Delta$ as the {\em graph} of $\Delta$ and denote it by $G(\Delta)$. 

The (abstract) $d$-simplex on the vertex set $V$ of size $d+1$ is the complex $\overline{V}:=\{\tau \ : \ \tau\subseteq V\}$. The boundary complex of $\overline{V}$, $\partial\overline{V}$, consists of all faces of $\overline{V}$ but $V$ itself.

A set $F\subseteq V$ is a {\em missing face} of $\Delta$ if $F$ is not a face of $\Delta$, but every proper subset $\sigma$ of $F$ is a face of $\Delta$. The {\em dimension} of a missing face $F$ is defined as $|F|-1$. The simplices are the only simplicial complexes that have no missing faces. The importance of missing faces of $\Delta$ is that they uniquely determine $\Delta$:  $F\subseteq V$ is a face of $\Delta$ if and only if no missing face of $\Delta$ is a subset of $F$. 

In this paper we only work with {\em simplicial} polytopes, that is, polytopes all of whose facets are geometric simplices. A simplicial $d$-polytope $P$ gives rise to an abstract simplicial complex $\partial P$ called the {\em boundary complex} of $P$: the faces of $\partial P$ consist of the empty set and the vertex sets of proper faces of $P$. We use the following convention: if $P$ is a $d$-polytope with $n$ vertices, we think of $\partial P$ as a simplicial complex on the vertex set $V(\partial P)=\{1,2,\ldots,n \}$ (or on any $n$ symbols), and we let $p:V(\partial P)\to\R^d$ be the map that takes each vertex $v$ to its position vector $p(v)$, so the vertices of $P$ are given by $p(v)$, $v\in V(\partial P)$. If $W$ is a subset of $V(\partial P)$ and $p$ is fixed or understood from context, we write  $\conv(W):=\conv(p(v) : v\in W)$ and $\Aff(W):=\Aff(p(v): v\in W)$. In particular, if $W$ is a face of $\partial P$, then $\conv(W)$ is a face of $P$.
To simplify notation, for a face that is a vertex or an edge, we sometimes write $v$ and $uv$, instead of $\{v\}$ and $\{u,v\}$, respectively, and we use $[a,b]$ as a shorthand for $\conv(\{a,b\})$.  

Let $F$ be a face of $\partial P$. The {\em star of $F$} and the  {\em link of $F$ in $\partial P$} are the following subcomplexes of $\partial P$:
$$\st(F)=\st_{\partial P}(F):=\{\sigma\in\partial P \ : \  \sigma\cup F\in\partial P\}\text{ and }\lk(F)=\lk_{\partial P}(F):= \{\sigma\in \st(F) \ : \ \sigma\cap F=\emptyset\}.$$ The link of $F$ in $\partial P$ is the boundary complex of a polytope. When $F=\{u\}$ is a vertex, one such polytope is obtained by {intersecting $P$ with a hyperplane that strictly separates $p(u)$ from the other vertices}; this polytope is called a {\em vertex figure of $P$ at $u$}  or a {\em quotient polytope of $P$ by $u$}. In general, a quotient of $P$ by face $F$ is obtained by iteratively taking vertex figures of polytopes at the vertices in $F$. Continuing with our convention from the previous paragraph, for any choice $Q$ of a quotient polytope of $P$ by $F$, the abstract simplicial complex $\partial Q$ coincides with $\lk_{\partial P}(F)$; in particular, the vertex set of $\partial Q$ is a subset of $V(\partial P)\setminus F$. The vertices of $Q$ itself are then of the form $q(v)$ for $v\in V(\partial Q)$ and an appropriate map $q$. This convention is handy for Corollary \ref{cor: cone lemma} and its applications.

If $\Gamma$ and $\Delta$ are simplicial complexes on disjoint vertex sets, their \textit{join} is the simplicial complex $\Gamma*\Delta = \{\sigma \cup \tau \ : \ \sigma \in \Gamma \text{ and } \tau \in \Delta\}$. When $\Gamma = \overline{u}$ consists of a single vertex, we write $u * \Delta$ to denote the \textit{cone} over $\Delta$ with apex $u$. Thus, $\st_{\partial P}(v)=v * \lk_{\partial P}(v)$.

Finally, for a simplical $d$-polytope or a $(d-1)$-dimensional simplicial complex $\Delta$, we define $f_i(\Delta)$ as the number of $i$-dimensional faces of $\Delta$, where $-1\leq i\leq d-1$. We also let $g_0(\Delta)=1$, $g_i(\Delta)=\sum_{k=0}^{i} (-1)^{i-k}\binom{d-k+1}{i-k}f_{k-1}(\Delta)$ for $1 \leq i\leq \lceil d/2 \rceil$, and $g_i(\Delta)=0$ for all other values of $i$.

\section{Basics on stress spaces and the rigidity theory of frameworks}
Here we review several notions and results related to (higher-dimensional) stresses and infinitesimal rigidity. For more details we refer the reader to \cite{Lee94,Lee96} and \cite{Tay-et-al-I, Tay-et-al}.
\subsection{Affine stresses on simplicial complexes}
Let $\Delta$ be a simplicial complex on the vertex set $V=V(\Delta)$.  A map $p: V(\Delta) \rightarrow \mathbb{R}^d$ is called a \textit{$d$-embedding} of $\Delta$. If $\Delta$ is a graph, then a $d$-embedding of $\Delta$ is usually called a {\em $d$-framework} or a {\em framework in $\R^d$}. In what follows, we fix an embedding $p$. 
We always assume that $\Aff(p(v) : v\in V)=\R^d$ and that for every face $F$ of $\Delta$, the points $\{p(v) : v\in F\}$ are affinely independent.

Let $X = \{x_v : v\in V\}$ be a set of variables and let $\R[X]$ be the polynomial ring over the real numbers in variables $X$. 
Each variable $x_v$ acts on $\R[X]$ by $\frac{\partial}{\partial{x_v}}$; for brevity, we will denote this operator by $\partial_{x_v}$. More generally, if $\ell(X)=\sum_{v\in V} \ell_v x_v$ is a linear form in $\R[X]$, then we define \begin{equation*}
	\begin{split}
		\partial_{\ell(X)} : \R[X]&\to\R[X], \\
		q &\mapsto \sum_{v\in V}\ell_v\cdot\partial_{x_v}q=\sum_{v\in V}\ell_v\frac{\partial q}{\partial{x_v}}.
	\end{split}
\end{equation*}

Given a $d$-embedding $p$ of $\Delta$, consider the $(d+1)\times |V|$ matrix whose columns are labeled by the vertices of $\Delta$: the column corresponding to $v$ consists of the vector $p(v)$ augmented by a one in the last position.  The $i$-th row of this matrix, $\boldsymbol\theta_i=[\theta_{iv}]_{v\in V}$, gives rise to a linear form $\theta_i=\sum_{v\in V}\theta_{iv} x_v$. In particular, $\theta_{d+1}=\sum_{v\in V} x_v$. We denote by $\Theta(p)$ or simply by $\Theta$ the sequence $(\theta_1,\ldots,\theta_d, \theta_{d+1})$  of these forms.

For a monomial $\mu\in \R[X]$, the {\em support} of $\mu$ is $\supp(\mu)=\{v\in V: \;x_v|\mu\}$.
A homogeneous polynomial $\lambda=\lambda(X)=\sum_\mu \lambda_\mu \mu\in\R[X]$ of degree $k$ is called an {\em affine  $k$-stress} (or simply a $k$-stress) on $(\Delta, p)$ if it satisfies the following conditions:
\begin{itemize}
	\item Every (non-zero) term $\lambda_\mu \mu$ of $\lambda$ is supported on a face of $\Delta$: $\supp(\mu)\in\Delta$, and
	\item $\partial_{\theta_i}\lambda=0$ for all $i=1,\ldots, d+1$.
\end{itemize}
The set of affine $k$-stresses on $\Delta$ forms a vector space, denoted $\Stress_k(\Delta,p)$ or $\Stress_k(\Delta)$ if $p$ is fixed or understood from context.

Abusing notation, we write $\lambda_F$ instead of $\lambda_\mu$ when $\mu$ is a squarefree monomial with $\supp(\mu)=F$. Note that a polynomial  $\lambda=\sum_{v\in V} \lambda_v x_v$ is an affine $1$-stress if and only if $(\lambda_v)_{v\in V}$ is an affine dependence of points $(p(v) : v\in V)$. More generally, for $k\geq 2$, an affine $k$-stress $\lambda$ is uniquely determined by its squarefree part $(\lambda_G)_{G\in \Delta}$ \cite[Theorems 9 and 11]{Lee96}. Furthermore, the squarefree part has a particularly nice geometric interpretation \cite[Theorem 10]{Lee96}. We summarize these results as follows: For a pair of faces $F\subset G\in \Delta$, where $|F|=k-1$ and $|G|=k$, we let $v_{F, G}$ be the unique vertex of $G$ that is not in $F$. Denote by  $\pi_{F, G}$ the {\em altitude vector} joining the projection of $p(v_{F, G})$ onto the affine hull of $p(F)$ to the point $p(v_{F, G})$. (For instance, when $k=2$, $\pi_{a, ab}$ is simply $p(b)-p(a)$.) Also denote by $\emb{0}$ the zero-vector in $\R^d$.

\begin{theorem}
	Let $k\geq 2$, let $\Delta$ be a simplicial complex, and let $p$ be a $d$-embedding of $\Delta$. 
	\begin{enumerate}
	\item If $\lambda\in \R[X]$ is an affine $k$-stress on $(\Delta,p)$, then for every $(k-2)$-face $F$ of $\Delta$, $\lambda$ satisfies the following {\em balancing condition at $F$}:
	\begin{equation} \label{bal-cond} \sum_{G:\; F\subset G \in \Delta, \; |G|=k} \lambda_G \; \pi_{F, G}=\mathbf{0}.\end{equation} 
	\item Every collection $(\lambda_G)_{G \in \Delta, \;|G|=k}$ of real numbers that satisfies these conditions 
	determines the squarefree part of an affine $k$-stress on $(\Delta,p)$; furthermore, such a stress is unique.
	\end{enumerate}
\end{theorem}

 For $k\geq 2$, the above theorem allows us to identify $\Stress_k(\Delta,p)$ with the kernel of a certain matrix $R_k(\Delta,p)$ called the {\em $k$-rigidity matrix} of $(\Delta,p)$.
The matrix $R_k(\Delta,p)$ is a $df_{k-2}(\Delta) \times f_{k-1}(\Delta)$ matrix; its columns are labeled by the $(k-1)$-faces of $\Delta$; its rows occur in blocks of size $d$ with each block labeled by a $(k-2)$-face of $\Delta$. The entry in the $(F,G)$-position is the altitude vector $\pi_{F,G}$ if $F\subset G$ and it is the zero vector otherwise.

We now discuss how the stress spaces of a complex $\Delta$ and the cone over $\Delta$ are related to each other. Assume that the vertex set of $\Delta$ is $[n]:=\{1,2,\ldots,n\}$, and let $\Gamma=0*\Delta$ be the cone over $\Delta$ with apex $0$. Let $p$ be a $d$-embedding of $\Gamma$ and let $p'$ be a $(d-1)$-embedding of $\Delta$. It is known that for appropriate choices of $p$ and $p'$, the stress spaces $\Stress_k(\Delta,p')$ and $\Stress_k(\Gamma, p)$ are isomorphic, see \cite[Theorem 7]{Lee96}. We will need the following more precise version of this result that is similar in spirit to  \cite[Theorem 7]{Lee96} and \cite[Claim 1 of Thm.~6.19]{Lee-notes}. We sketch the proof for completeness. 

\begin{lemma} \label{cone-lemma1}
Let $\Delta$ be a simplicial complex with $V(\Delta)=[n]$ and let $\Gamma=0*\Delta$. Consider a $d$-embedding $p$ of $\Gamma$ such that $p(0)=\emb{0}$ and for all $i\in [n]$, $p(i)=\left[ \begin{array}{ll} \emb{v}_i \\ a_i\end{array}\right]$ where $\emb{v}_i \in \R^{d-1}$ and $a_i\in \R\backslash\{0\}$. Define the $(d-1)$-embedding $p'$ of $\Delta$ by $p'(i)=\frac{1}{a_i} \emb{v}_i \in\R^{d-1}$ for all $i\in [n]$. If $\lambda=\lambda(x_0,x_1,\ldots,x_n)$ is a homogenous polynomial of degree $k$, express it as $\lambda=\sum_{j=0}^k x_0^j\cdot\lambda_j(x_1,\ldots,x_n)$.
\begin{enumerate}
\item Let $\lambda\in\Stress_k(\Gamma,p)$. Then $\bar{\lambda}:=\lambda_0(a_1x_1, \dots, a_nx_n)$ is in $\Stress_k(\Delta,p')$.
\item The linear map 
			\begin{equation*}
				\begin{split}
					\Stress_k(\Gamma, p) &\to \Stress_k(\Delta, p')\\
					\lambda(x_0, x_1, \dots, x_n) &\mapsto \lambda_0(a_1x_1, \dots, a_nx_n)
				\end{split}
			\end{equation*} is an isomorphism. In particular, every affine $k$-stress $\omega'$ on $(\Delta,p')$ lifts to an affine $k$-stress $\omega$ on $(\Gamma,p)$ with the property that $\omega'_F=(\prod_{i\in F}a_i)\omega_F$ for every $(k-1)$-face $F\in \Delta$.
\end{enumerate}
	\end{lemma}
\begin{proof} 
Consider the linear forms $\Theta=\Theta(p)$ and $\Theta'=\Theta'(p')$ used in the definition of $\Stress_k(\Gamma, p)$ and $\Stress_k(\Delta, p')$. Note that the way $p'$ is related to $p$ implies that
\[\theta_{i0}=0 \mbox{ for all $i\leq d$; } \, \theta'_{ij}=\frac{1}{a_j}\theta_{ij} \mbox{ for all $1\leq i\leq d$ and $j\in[n]$}.\]
In particular, $\theta'_d=\sum_{i\in [n]}x_i$.
A straightforward computation then shows that for $\lambda\in\Stress_k(\Gamma,p)$, 
\begin{eqnarray*}
\partial_{\theta'_i}(\bar{\lambda})&=&\big( \partial_{\theta_i}(\lambda_0)\big) (0, a_1x_1, \dots, a_nx_n) \quad \quad \mbox{ and } \\
 0&=&\big(\partial_{\theta_i} (x_0^j\cdot\lambda_j)\big) (0, a_1x_1, \dots, a_nx_n) \quad \mbox{for any $j\geq 1$}.
\end{eqnarray*}
Consequently, 
$$ \partial_{\theta'_i}(\bar{\lambda})=\big( \partial_{\theta_i} (\lambda_0+\sum_{1\leq j\leq k} x_0^j\cdot\lambda_j)\big) (0, a_1x_1, \dots, a_nx_n)= \big( \partial_{\theta_i} \lambda\big)(0, a_1x_1, \dots, a_nx_n)=0 \quad \forall\; i\leq d,$$
and hence $\bar{\lambda}\in \Stress_k(\Delta,p')$.

For part 2, use that if $\lambda\in\Stress_k(\Gamma,p)$, then 
\[
0=\partial_{\theta_{d+1}} \lambda = \partial_{\theta_{d+1}} \big(\lambda_0 + \sum_{1\leq j\leq k}x_0^j \lambda_j\big)
=\sum_{0\leq j\leq k-1} x_0^j\big( \partial_{\theta'_d} \lambda_j + (j+1)\lambda_{j+1}\big).
\] 
For this to happen, we must have  $\lambda_{j+1}=-\frac{1}{j+1} \partial_{\theta'_d} \lambda_j$ for all $0\leq j\leq k-1$. Thus all $\lambda_j$ are determined by $\lambda_0$, which implies the injectivity of the map $\lambda(x_0, x_1, \dots, x_n) \mapsto \lambda_0(a_1x_1, \dots, a_nx_n)$. To prove its surjectivity, for $\omega'\in \Stress_k(\Delta,p')$, take $\omega_0(x_1,\ldots,x_n):=\omega'(x_1/a_1,\ldots,x_n/a_n)$ and then define $\omega_{j+1}(x_1,\ldots,x_n)$ inductively by $\omega_{j+1}:=-\frac{1}{j+1}\partial_{\theta'_d}\omega_j$ for $0\leq j\leq k-1$. Reversing the above computations shows that $\omega:= \sum_{j=0}^k x_0^j\omega_j$ is in $\Stress_k(\Gamma,p)$.
\end{proof}

In what follows, if $P\subset \R^d$ is a simplicial $d$-polytope, we always use the natural $d$-embedding $(\partial P, p)$ of $P$ and write it as $(P, p)$, where $p(v)$ is the position vector of $v$. We also write $\Stress_k(P)$ and $G(P)$ instead of $\Stress_k(\partial P,p)$ and $G(\partial P)$, respectively. Applying Lemma \ref{cone-lemma1}(2) to subcomplexes of boundary complexes of polytopes yields the following result that is at the core of the approach we will be taking in Section~4.

\begin{corollary}\label{cor: cone lemma}
	Let $P$ be a simplicial $d$-polytope with its natural embedding $p$, let $\tau\in \partial P$ be a face, and let $Q \subset \R^{d-|\tau|}$ be a quotient polytope of $P$ by $\tau$ given with its natural embedding $q$. Let $\Delta$ be a simplicial complex on $V(\partial Q) =V(\lk_{\partial P}(\tau))$ with $
	\skel_{k-1}(\overline{\tau}*\Delta)\subseteq \partial P$. Then for every $k$-stress $\bar{\lambda}$ on  $(\Delta, q)$ there exists a $k$-stress $\lambda$ on $(\overline{\tau}*\Delta, p)$ with the property that for each $(k-1)$-face $F \in \Delta$, the real numbers  $\bar{\lambda}_F$ and $\lambda_F$ have the same sign, i.e., they are both positive or both negative or both zeros. 
\end{corollary}

\begin{proof}
	It suffices to prove the statement in the case that $\tau$ is a vertex. Since the space of affine stresses is unaffected by Euclidean motions and scalings, we can always assume that $p(\tau)$ is the origin, and that the hyperplane $H$ that $Q$ lies in, i.e., the hyperplane we use to separate $p(\tau)$ from the rest of the vertices is $H=\{(t_1,\ldots,t_d)\in\R^d : t_d=1\}$, and hence that for each vertex $v\neq \tau$, the last coordinate of $p(\tau)$ is strictly greater than 1. The isomorphism provided by Lemma \ref{cone-lemma1}(2) then lifts a $k$-stress  $\bar{\lambda}$ on $(\Delta, q)$ to a $k$-stress $\lambda$ on $(\overline{\tau}*\Delta,p)$ that has the desired property.
\end{proof}

\subsection{Infinitesimal rigidity of frameworks}

Let $(G, p)$ be a $d$-framework. Recall our assumption that $\Aff(p(v) : v\in V(G))$ is $\R^d$ (i.e., this framework does not lie in a hyperplane of $\R^d$). The left kernel space of the $2$-rigidity matrix $R_2(G, p)$ is called the {\em infinitesimal motion space} of $G\subset \R^d$. Since all Euclidean motions of $\R^d$ induce infinitesimal motions of $G$, it follows that the dimension of this space is at least ${d+1 \choose 2}$. We say that $(G, p)$ is {\em infinitesimally rigid} in $\R^d$ if the dimension of its infinitesimal motion space is exactly $\binom{d+1}{2}$. Basic linear algebra then yields:
\begin{theorem}\label{thm: rigidity equivalence}
	Let $(G, p)$ be a $d$-framework with $f_0$ vertices and $f_1$ edges. The following statements are equivalent:
	\begin{enumerate}
		\item $G$ is infinitesimally rigid in $\R^d$.
		\item The rank of $R_2(G, p)$ is $df_0-\binom{d+1}{2}$.
		\item The dimension of $\Stress_2(G, p)$ is $f_1-df_0+\binom{d+1}{2}$.
	\end{enumerate}
\end{theorem}

One immediate and well-known corollary of Theorem \ref{thm: rigidity equivalence} we will use is   

\begin{corollary}\label{cor: add missing edge}
	Let $d\geq 3$. Let $(G, p)$ be a $d$-framework and let $e$ be a missing edge of $G$. If $(G, p)$ is infinitesimally rigid in $\R^d$, then there exists an affine $2$-stress $\lambda$ on $(G \cup \{e\}, p)$ with $\lambda_e\neq 0$.
\end{corollary}

The following fundamental theorem is due to Whiteley \cite{Whiteley-84}. 

\begin{theorem}\label{Whiteley-thm}
	Let $d\geq 3$ and let $P\subset \R^d$ be a simplicial $d$-polytope. The graph of $P$ with its natural embedding is infinitesimally rigid in $\R^d$. In particular, $g_2(P)=\dim \Stress_2(P)$.
\end{theorem}

The case $d=3$ of this theorem is due to Dehn and is known as Dehn's lemma; we shall review its proof in Section 5. Whiteley's proof for $d\geq 4$ is by induction on $d$. One crucial ingredient is the following result about stars of faces that follows from the Cone Lemma (applied to vertex figures) 
and Theorem \ref{thm: rigidity equivalence}.

\begin{lemma}\label{cone-lemma2}
	Let $d \geq 4$ and let $P$ be a simplicial $d$-polytope with its natural embedding $p$ in $\R^d$. Then for every face $\tau$ of $\partial P$ with $0\leq \dim \tau \leq d-4$, the framework $(G(\st_{\partial P}(\tau)), p)$ is infinitesimally rigid in $\R^d$.
\end{lemma}

The other ingredient of Whiteley's proof is the Gluing Lemma, which allows us to form larger infinitesimally rigid frameworks from the stars of faces in a polytope.
\begin{lemma} {\rm{(The Gluing Lemma \cite[Thm.~2]{AsimowRothII} and \cite[Cor.~6.12]{Lee-notes})}} \label{gluing-lemma}
	Let $G$ and $G'$ be graphs, and let $(G \cup G',p)$ be a $d$-framework.  If $(G, p)$ and $(G', p)$ are infinitesimally rigid in $\R^d$ and have $d$ affinely independent vertices in common (i.e., the framework $(G \cap G', p)$ affinely spans a subspace of dimension at least $d-1$), then $(G \cup G', p)$ is infinitesimally rigid in $\R^d$.
\end{lemma}

We end this section mentioning the celebrated $g$-theorem \cite{McMullen96, Stanley80} that provides a far reaching generalization of Dehn's lemma. Stated in the language of stresses it asserts the following.

\begin{theorem}
Let $P\subset\R^d$ be a simplicial $d$-polytope. Then $g_k(P)=\dim \Stress_k(P)$ for all $k$.
\end{theorem}

\section{Several variations of Kalai's conjecture}
In this section we propose and discuss several conjectures of increasing strength each of which implies Conjecture \ref{conj:Kalai}.
Our approach is motivated by the following toy example. Consider two sets $\sigma=\{1,2,\ldots,i, i+1\}$ and $\tau=\{i+2,\ldots,d+2\}$ where $2\leq i\leq d/2$, and let $\Delta$ be $\partial\overline{\sigma}*\partial\overline{\tau}$. Then $\Delta$ is a simplicial sphere realizable as the boundary complex of a simplicial $d$-polytope. Let $P\subset \R^d$ be any such polytope. Since $\sigma$ and $\tau$ are missing faces of  $\partial P=\Delta$ and since their union is the entire vertex set, the convex hulls of $p(\sigma)$ and $p(\tau)$ must intersect in their relative interiors. (Here, as always, $p(v)$ is the position vector of vertex $v$ in $P$.) Thus there exist positive coefficients $c_\ell\in \R$ such that 
$$\sum_{\ell=1}^{i+1} c_\ell \; p(\ell)- \sum_{\ell=i+2}^{d+2} c_\ell \; p(\ell)=\emb{0}, \quad  \sum_{\ell=1}^{i+1} c_\ell=\sum_{\ell=i+2}^{d+2} c_\ell=1.$$
Hence $\lambda:=\sum_{\ell=1}^{i+1} c_\ell x_\ell- \sum_{\ell=i+2}^{d+2} c_\ell x_{\ell}$ is an element of $\Stress_1(P)$. (Here, following the notation of Section 3.1, $x_{\ell}$ denotes the variable corresponding to vertex $\ell$.) 

Since every $i$-subset of $\sigma\cup\tau$ forms a face of $\partial P$, it follows that $\lambda^k$ is in $\Stress_k(P)$ for all $1\leq k\leq i$. On the other hand, $P$ is $i$-neighborly and has $d+2$ vertices. In particular, for each $1\leq k\leq i$, $g_k(P)=\binom{(d+2)-d+(k-2)}{k}=1$. (This is well-known, see \cite[Section 8.4]{Ziegler}, and also easily follows  by direct computation.) Hence $\dim \Stress_k(P)=1$. This implies that $\lambda^k$ spans $\Stress_k(P)$.  An important thing to observe now is that by our definition of $\lambda$, the values $((\lambda^k)_G)_{G\in\Delta,  |G|=k}$ have the following property:  if $F$ is a $(k-1)$-subset of the missing face $\sigma$, then for every $\ell\in \sigma\setminus F$, $(\lambda^k)_{F\cup \ell}$ is positive, while for every $\ell\in \tau$, $(\lambda^k)_{F\cup \ell}$ is negative. A surprising aspect of this observation is that it holds for {\bf any} polytope $P$ whose boundary complex is $\Delta$! 

This observation suggests that the collection of sign vectors of $k$-stresses on $P$ may contain enough information to identify the missing faces of $P$. We are thus led to the following definition.

\begin{definition} Let $P$ be a simplicial $d$-polytope with its natural embedding $p$. For an affine $k$-stress $\lambda$ on $(P, p)$ and a $(k-1)$-face $G$ of $P$, let $$\sign(\lambda_G)=
	\begin{cases}
		+ & \mbox{if }\lambda_G>0 \\
		- & \mbox{if }\lambda_G<0 \\
		0 & \mbox{if }\lambda_G=0.
	\end{cases}$$
	Define $\V_k(P)=\{(\sign(\lambda_G))_{G\in \partial P, \;|G|=k}: \lambda\in \Stress_k(P)\}$. Thus $\V_k(P)$ is the collection of sign vectors of the squarefree parts of $k$-stresses on $P$.
\end{definition}

With this definition in hand, we propose the following strengthening of Conjecture \ref{conj:Kalai}.
\begin{conjecture}\label{conj1}
	Let $k\geq 2$ and $d\geq 2k$. Let $P\subset \R^d$ be a simplicial $d$-polytope. The $(k-1)$-skeleton of $\partial P$ and the set $\V_k(P)$ determine the entire complex $\partial P$.
\end{conjecture}

We are about to strengthen this conjecture even more. This requires the following simple lemma.

\begin{lemma}\label{lm: missing face condition}
	Let $k\geq 2$ and $d\geq 2k$. Let $P\subset \R^d$ be a simplicial $d$-polytope, and let $M\subset V(\partial P)$ be a set of size $\geq k$. Assume there is a $(k-2)$-face $F\subset M$ of $\partial P$ and an affine $k$-stress $\lambda$ on $P$ such that for every $(k-1)$-face $G=F\cup v_{F,G}$ of $\partial P$ with $v_{F,G}\notin M$,  $\lambda_G\leq 0$ and at least one of these numbers $\lambda_G$ is negative. Then $M$ is not a face of $\partial P$. 
\end{lemma}
\begin{proof}
	Assume that $M$ is a face. Let $H=\{x: x\cdot b=\alpha\}$ be a hyperplane that defines $M$, that is, we assume that $P\subseteq \{x:x \cdot b\geq \alpha\}$ and that $H\cap P=\conv(M)$. (Here $x\cdot b$ denotes the dot product.) By the balancing condition on $\lambda$ at $F$, see \eqref{bal-cond}, $$\mathbf{0}=\sum_{G: \; F\subset G,\; |G|=k}\lambda_G\pi_{F, G}=\sum_{G: \;v_{F, G}\in M}\lambda_G\pi_{F, G} + \sum_{G: \;v_{F, G}\notin M}\lambda_G\pi_{F, G}.$$ 
	Computing the dot product with $b$ and keeping in mind that $H$ defines $M$, we obtain 
	\begin{equation*}
		\begin{split}
			0=\sum_{G: \;v_{F, G}\in M}\lambda_G(\pi_{F, G}\cdot b) + \sum_{G: \;v_{F, G}\notin M}\lambda_G(\pi_{F, G}\cdot b)=\sum_{G: \;v_{F, G}\notin M}\lambda_G(\pi_{F, G}\cdot b)<0,
		\end{split}
	\end{equation*}
	which is a contradiction. (In the last step we used that for faces $G$ with $v_{F,G}\notin M$,  $\pi_{F,G}\cdot b>0$ while $\lambda_G\leq 0$ and at least one of the numbers $\lambda_G$ is negative.)
\end{proof}

If the converse of Lemma \ref{lm: missing face condition} holds, then the sign vectors in $\V_k(P)$ completely determine the set of all missing faces $M$ with $k\leq \dim M\leq d-1$. 
This motivates the following conjecture.

\begin{conjecture}\label{conj2}
Let $k\geq 2$ and $d\geq 2k$. Let $P$ be a simplicial $d$-polytope and let $M$ be a missing face of $\partial P$ with $k\leq \dim M \leq d-k$. Then there exists a $(k-2)$-face $F\subset M$ and an affine $k$-stress $\lambda$ on $P$ with the property that 
\begin{enumerate}
\item[$(*)$]for every $(k-1)$-face $G=F\cup v_{F,G}$ of $\partial P$, $\lambda_G > 0$ if $v_{F,G}\in M$ while $\lambda_G\leq 0$ if $v_{F,G}\notin M$. 
\end{enumerate}
\end{conjecture}

\begin{lemma}
	Conjecture \ref{conj2} implies Conjecture \ref{conj1}.
\end{lemma}
\begin{proof}
Assume that a simplicial $d$-polytope $P$ satisfies the statement of Conjecture \ref{conj2}. Use the complex $\skel_{k-1}(\partial P)$ and the set $\V_k(P)$ to find all subsets $M\subset V(\partial P)$, $k+1\leq |M|\leq d-k+1$ such that (1) $\skel_{k-1}(\overline{M})\subseteq \skel_{k-1}(\partial P)$, and (2) there is a $(k-1)$-subset $F$ of $M$ and a $k$-stress $\lambda$ such that the triple $(M,F,\lambda)$ satisfies condition $(*)$ of Conjecture \ref{conj2}. Let $\M$ be the collection of all such $M$. Note that if $M\in \M$ and 
$(M,F,\lambda)$ is a triple satisfying condition $(*)$, then the balancing condition on $\lambda$ at $F$ guarantees that $\lambda_G$ is negative for at least one face $G=F\cup v_{F,G}$. (To see this, perform the same computation as in the proof of Lemma \ref{lm: missing face condition}, but using a hyperplane $H'$ that defines $F$.) Hence by Lemma \ref{lm: missing face condition}, no element $M$ of $\M$ is a face of $\partial P$. Our assumption that Conjecture \ref{conj2} holds then implies that the minimal (w.r.t~inclusion) elements of $\M$ are precisely the missing faces of $\partial P$ of dimensions between $k$ and  $d-k$. This allows us to reconstruct $\skel_{d-k}(\partial P)$. The result of Perles and Dancis, \cite{Dancis}, then allows us to reconstruct the entire complex $\partial P$. 
\end{proof}

We now propose another conjecture that implies Conjecture \ref{conj2} and hence also Conjecture \ref{conj1}.

\begin{conjecture}\label{conj3}
	Let $k\geq 2$ and let $P\subset \R^d$ be a simplicial polytope of dimension $d\geq 2k-1$. If $G$ is a missing $(k-1)$-face of $\partial P$ and $F$ is a $(k-1)$-subset
	of $G$, then there exists an affine $k$-stress $\lambda$ on $(\partial P\cup \{G\}, p)$ such that $\lambda_{G}>0$ and $\lambda_{F\cup u}\leq 0$ for every $(k-1)$-face $F\cup u$ of $\partial P$. 
\end{conjecture}
\noindent Note that the balancing condition on $\lambda$ at $F$ implies that $\lambda_{F\cup u}<0$ for at least one face $F\cup u\in\partial P$.

\begin{lemma}
	Conjecture \ref{conj3} implies Conjecture \ref{conj2}.
\end{lemma}
\begin{proof}
Let $P$ be a simplicial $d$-polytope, let $M=\{x_0, x_1, \dots, x_{\ell}\} \subset V(\partial P)$ be a missing face of $\partial P$, where $k\leq \ell\leq d-k$, and let $F=\{x_1, \dots, x_{k-1}\}$. We want to find an affine $k$-stress $\lambda$ on $P$ so that $(M,F,\lambda)$ satisfies condition $(*)$ of  Conjecture \ref{conj2}. To start, note that $F$,
$G:=F\cup x_0$, and $M\backslash G=\{x_k,\ldots,x_\ell\}$ are faces of $\partial P$, but $G$ is a missing face of $\lk_{\partial P}(M\backslash G)$. Let $Q$ be a quotient polytope of $P$ by $M\backslash G$ (with its natural embedding $q$). Then $Q$ has dimension $d-1-(\ell-k)\geq 2k-1$ and $G$ is a missing $(k-1)$-face of $\partial Q=\lk_{\partial P}(M\backslash G)$. 

Our assumption that Conjecture \ref{conj3} holds implies the existence of a $k$-stress $\bar{\lambda}$ on $(\partial Q \cup \{G\}, q)$ such that $\bar{\lambda}_{G}>0$ while $\bar{\lambda}_{F\cup u}\leq 0$ for every $(k-1)$-face $F\cup u\in \partial Q$. Applying Corollary \ref{cor: cone lemma} to this stress $\bar{\lambda}$, provides us with a $k$-stress $\lambda$ on $\st_{\partial P}(M\backslash G) \cup \{G\} \subset \partial P$ with the property that for every $(k-1)$-face $\tau\in \lk_{\partial P}(M\backslash G)\cup \{G\}$, $\bar{\lambda}_{\tau}$ and $\lambda_{\tau}$ have the same sign. 
In particular, $\lambda$ is a $k$-stress on $P$ that satisfies $\lambda_{G}>0$ and $\lambda_{F\cup u}\leq 0$ for all faces $F\cup u\in\partial P$ with $u\in V(\partial P)\backslash M$. (Note that $\lambda_{F\cup u}=0$ if $u\notin \st_{\partial P}(M\backslash G)$.)

	Let $M':=F \cup \{u\in M\backslash F :  \lambda_{F\cup u}> 0\}$. To complete the proof, it remains to show that $M'=M$. To do so, note that $M'$ contains $G=F\cup x_0$ and that the triple $(M',F, \lambda)$ satisfies all the assumptions of Lemma \ref{lm: missing face condition} (including the assumption that $\lambda_{F\cup v}$ is strictly negative for at least one face $F\cup v$, $v\notin M'$; this follows from the balancing condition at $F$ and the fact that $\lambda_{G}>0$).
	Hence by Lemma \ref{lm: missing face condition}, $M'$ is not a face. As $M'$ is contained in the missing face $M$ of $\partial P$, it follows that $M=M'$. 
\end{proof}

\section{Two cases of Kalai's conjecture}
The goal of this section is twofold. We first prove that the strongest of the three conjectures discussed in the previous section, namely Conjecture \ref{conj3}, holds for the case of $k=2$ and all simplicial polytopes of dimension $\geq 3$. We then show that Conjecture \ref{conj2} holds for all $k$-neighborly polytopes for an arbitrary $k$. This establishes the validity of Conjecture \ref{conj:Kalai} in these two cases.

\subsection{The case of $k=2$}
To verify the $k=2$ case of Conjecture \ref{conj3} for simplicial polytopes of dimension $d\geq 3$, we separately treat the cases of $d=3$ and $d>3$. The $d=3$ case is established by the following result that can be considered as an extension of Dehn's lemma. The proof is almost identical to one of the proofs of Dehn's lemma (see  \cite[Theorem 6.17]{Lee-notes} and also \cite[Ch.~26.3 \& 32.3]{Pak-book}).

\begin{lemma}\label{lm: identify missing 2-faces}
	Let $P$ be a simplicial 3-polytope and let $ab$ be a missing edge of $\partial P$. Then there is a unique affine 2-stress $\lambda$ on $G(P) \cup ab$ such that $\lambda_{ab}=1$; this stress satisfies $\lambda_{e}\leq 0$ for every edge $e\in G(P)$ that is incident to $a$ or $b$.
\end{lemma}
\begin{proof}
Since simplicial 3-polytopes are infinitesimally rigid and do not support nontrivial affine $2$-stresses, it follows from Corollary \ref{cor: add missing edge} that there exists a $2$-stress $\lambda$ supported on $G(P)\cup ab$ with $\lambda_{ab}\neq 0$ and it is unique up to scalar multiplication. So we can assume that $\lambda_{ab}=1$. We label each edge $e$ of $G(P)$ with $+, -, 0$ according to $\sign(\lambda_e)$. 

We follow the notation of \cite[p.~251]{Pak-book}  and give a sketch of the proof below.
Given a vertex $v\in V(P)$, consider the labels of the edges of $P$ containing $v$ written in the cyclic order induced by $\lk_{\partial P}(v)$ and ignoring the zero labels. Denote by $m_v$ the number of sign changes at $v$. Let $N=\sum_{v\in V} m_v$. By \cite[Lemma 32.3]{Pak-book}, for $v\neq a,b$, $m_v \geq 4$ unless all labels around $v$ are zeros. On the other hand, $N$ can also be computed as $N=\sum_{F} n_F$ where the sum is over $2$-faces (i.e., triangles) and $n_F$ is the number of sign changes around $F$; in particular, $n_F\leq 2$ for every $2$-face~$F$. 

The lemma will follow if we prove that $m_a=m_b=0$ (as the balancing condition at $a$ would then imply that $\lambda_e\leq 0$ for all $e\in G(P)$ with $a\in e$, and similarly for $b$). There are two cases to consider.  If no vertex has  {\em all} edges incident to it labeled $0$, then $$ 4(f_0-2) \leq \sum_{v\in V} m_v \leq N \leq 2f_2=2(2f_0-4), $$ which forces $m_a=m_b=0$. Otherwise, some vertices have all edges incident to them labeled $0$. Consider the graph $G'$ obtained from $G(P)$ by first removing all such vertices, and then adding edges to triangulate all resulting non-triangular $2$-faces and labeling the new edges $0$. The computation as above applies to $G'$ and implies that $m_a=m_b=0$ in $G'$. Hence $m_a=m_b=0$ also in $G(P)$.
\end{proof}

Our next goal is to prove Conjecture \ref{conj3} in the case of $k=2$ and $d>3$. This will require a bit of preparation. The proof idea in this case is based on Balinski's theorem \cite{Balinski}, see also \cite[Section 3.5]{Ziegler}, asserting that every $d$-polytope $P$ is {\em $d$-connected}.
We need a couple of extensions of Balinski's theorem. We only sketch the proofs as they are easy consequences of the proof of the theorem.
\begin{lemma}\label{lm: Balinski}
	Let $P$ be a $d$-polytope (not necessarily simplicial) and let $W$ be a set of vertices of $P$ whose removal disconnects $G(P)$. Then $\Aff(W)$ is at least $(d-1)$-dimensional.
\end{lemma}
\begin{proof}
Assume $\dim \Aff(W)\leq d-2$. Choose any $v_0\in V(P)\backslash W$ and a hyperplane $H$ in $\R^d$  that contains $\Aff(W\cup v_0)$. Let $\phi: \R^d \to \R$ be a non-zero linear function that is zero on $H$. If the maximum value $\phi_{\max}$ of $\phi$ on $P$ is positive, it is attained on the vertices of a face $F_{\max}$ of $P$.
In this case, the graph $G_{\max}:=G(F_{\max})$ is a connected graph contained in $G(P)\backslash W$. Similar assertions hold for $G_{\min}:=G(F_{\min})$ (assuming $\phi_{\min}$ is negative).

Now, for each vertex $v\in V(P)\backslash W$ with $\phi(v)\geq 0$ (including $v_0$), the simplex algorithm gives a $\phi$-increasing path connecting $v$ to some vertex of $G_{\max}$. Such a path lies in $G(P)\backslash W$. Similarly, each $v\in V(P)\backslash W$ with $\phi(v)\leq 0$ (including $v_0$) is connected to some vertex of $G_{\min}$ by a $\phi$-decreasing path that lies in $G(P)\backslash W$. It follows that $G(P)\backslash W$ is connected, which is a contradiction. (This proof is easily adjusted to the case where $\phi_{\max}$ or $\phi_{\min}$ is zero.)
\end{proof}

\begin{lemma}\label{lm: intersecting convex hulls}
	Let $P$ be a simplicial $d$-polytope and let $ab$ be a missing edge of $\partial P$. Let $C$ be the collection of vertices $c$ in $\lk_{\partial P}(a)$ with the property that there is a path $\Pi_c$ from $c$ to $b$ in $G(P)$ such that no internal vertex of $\Pi_c$ is in $\lk_{\partial P}(a)$. Then $\conv(C)\cap [a,b]\neq \emptyset$.
\end{lemma}
\begin{proof}
	 Assume that $\conv(C)\cap [a,b]= \emptyset$, and let $H$ be a hyperplane that separates $\conv(C)$ from $[a,b]$ and does not contain any vertices of $G(P)$. W.l.o.g., $[a,b]\subseteq \intr(H^+)$ and $\conv(C)\subseteq \intr(H^-)$. Then $H$ defines a linear function $\phi$ that is zero on $H$ and positive on $\intr(H^+)$. As in the proof of Lemma \ref{lm: Balinski}, each vertex $v$ in $\intr(H^+)$ can be connected by a $\phi$-increasing path to $G_{\max}$, which implies that the restriction $G'$ of $G(P)$ to the vertices in $\intr(H^+)$ is a connected subgraph of $G(P)$. In particular, $G'$ contains a path $(a=a_0, a_1, \dots, a_\ell=b)$, where $\ell\geq 2$. Let $i_0$ be the largest index such that $a_{i_0}\in \lk(a)$. It exists since $a_1\in \lk(a_0)=\lk(a)$. Thus $\Pi_{a_{i_0}}:=(a_{i_0}, a_{i_0+1}, \dots, a_\ell=b)$ is a path as in the statement of the lemma and hence $a_{i_0}\in C$. But $\conv(C)\subseteq \intr(H^-)$, which is impossible because this entire path is contained in $\intr(H^+)$. This gives us a desired contradiction.
\end{proof}

We are now ready to prove Conjecture \ref{conj3} in the case of $k=2$ and $d>3$. Specifically, we prove

\begin{lemma} \label{lm:k=2,d>3}
Let $d\geq 4$, let $P$ be a simplicial $d$-polytope, and let $ab$ be a missing edge of $\partial P$. Then there is a 2-stress $\lambda$ on $G(P)\cup ab$ such that $\lambda_{ab}=1$ and $\lambda_e\leq 0$ for every edge $e\in G(P)$ that is incident to $a$.
\end{lemma}

\begin{proof}
	Let $C\subseteq V(\lk_{\partial P}(a))$ be as in the statement of Lemma \ref{lm: intersecting convex hulls}. By definition of $C$, $C$ separates $a$ from $b$. Hence by Lemmas \ref{lm: Balinski} and \ref{lm: intersecting convex hulls}, $\dim\Aff(C)\geq d-1$ and $\conv(C)\cap [a,b]\neq \emptyset$. Either $\conv(C)$ is $(d-1)$-dimensional, or since the vertices $p(a), p(b)$ lie outside of $\conv(C)$, $\conv(C)\cap [a,b]$ must contain a boundary point of $\conv(C)$. In either case, by Carath\'eodory's theorem there exists a $d$-subset $C'$ of $C$ such that the points of $p(C')$ are affinely independent and $\conv(C')\cap [a,b]\neq \emptyset$.

    For each $c\in C'$, consider a path $\Pi_c$ from $c$ to $b$ as in Lemma \ref{lm: intersecting convex hulls}, and let $\Pi'_c$ be the path obtained from $\Pi_c$ by deleting the initial vertex $c$. Define $$K:=\bigcup_{c\in C'}\bigcup_{v\in V(\Pi'_c)} \st_{\partial P}(v).$$ Note that all paths $\Pi'_c$ share a common end-point $b$. Hence the union of paths $\Pi'_c$ over $c\in C'$ is a connected subgraph of $G(P)$. Since $d\geq 4$, all vertex stars of $P$ are infinitesimally rigid in $\R^d$. Hence by the Gluing Lemma, $(K,p)$ is infinitesimally rigid. Furthermore, $a\notin V(K)$ but $C'\subseteq V(K)$. (This is because, each $c\in C'$ is in the star of its neighbor on the path $\Pi_c$.)
	
	Add to the graph of $K$ the $d$ edges $\{ac: c\in C'\}$ to obtain a new graph $G' \subseteq G(P)$. Since the $d$ vertices from $C'$ are affinely independent and $p(a)\notin \Aff(p(C'))$, it follows that the graph $(G',p)$ is also infinitesimally rigid. Note that $ab$ is a missing edge of $G'$, and so by Corollary \ref{cor: add missing edge}, the graph $(G'\cup ab, p)$ supports a $2$-stress $\lambda$ with $\lambda_{ab}=1$. Now, in $G'\cup ab$, the vertex $a$ is incident to exactly $d+1$ edges, namely, $ab$ as well as $ac$ for $c\in C'$. Since $p(C') \subset \R^d$ is affinely independent and has size $d$, there is a unique way to write  $$\lambda_{ab}\pi_{a, ab}=p(b)-p(a)=\sum_{c\in C'} k_c (p(c)-p(a))=\sum_{c\in C'} k_c\pi_{a, ac}.$$ As $[a,b]\cap \conv(C')\neq \emptyset$, all coefficients $k_c$ in this expression are nonnegative. Hence in the balancing condition at $a$, $\lambda_{ac}=-k_c\leq 0$ for all $c\in C'$. Thus $\lambda$ is a desired stress on $G'\cup ab \subseteq G(P) \cup ab$. 
\end{proof}

We conclude from Lemmas \ref{lm: identify missing 2-faces} and \ref{lm:k=2,d>3} the following:
\begin{theorem} \label{thm:k=2}
Conjecture \ref{conj3} holds in the case of $k=2$. 
\end{theorem}

\subsection{$k$-neighborly polytopes}
We are now in a position to prove Conjecture \ref{conj2} for all $k$-neighborly polytopes; thus verifying Conjecture \ref{conj:Kalai} for this family. In fact, the following theorem, whose proof is almost identical to that of our toy example in Section 4,  establishes a somewhat stronger result.  

\begin{theorem}  \label{thm:k-neighb}
	Let $k\geq 2$ and $d\geq 2k$. Let $P$ be a $k$-neighborly simplicial $d$-polytope. If $M$ is a missing face of $\partial P$, then there is an affine $k$-stress $\lambda=\sum_{\mu}\lambda_\mu \mu$ such that for every $k$-subset $G$ of $V=V(\partial P)$,   
	 $$ \lambda_G>0 \mbox{ if } G\subset M \quad \mbox{and} \quad (-1)^{|(V\backslash M)\cap G|}\lambda_G \geq 0 \mbox{ if }G\not\subset M.$$
\end{theorem}
\begin{proof}
	 Since $M$ is a missing face of $\partial P$, the relative interior of $\conv(M)$ intersects with $\conv(V\backslash M)$. In other words, there exist nonnegative coefficients $a_v$ such that $$\sum_{v\in M}a_v p(v) = \sum_{v\in V\backslash M} a_vp(v), \quad \sum_{v\in M} a_v=\sum_{v\in V\backslash M} a_v=1, \quad \text{and } a_v>0 \text{ if }v\in M.$$ This implies that $\phi:=\sum_{v\in M}a_v x_v -\sum_{v\in V\backslash M} a_vx_v$ is an affine 1-stress on $P$, i,e., $\partial_{\theta_i} (\phi)=0$ for all $1\leq i\leq d+1$, where $\Theta(p)=(\theta_1, \dots, \theta_{d+1})$ is the set of linear forms associated with $p$.
	
	Let $\lambda:=\phi^k$. Since $P$ is $k$-neighborly, all monomials of $\lambda$ are supported on the faces of $\partial P$. Furthermore, since  $\partial_{\theta_i} \phi=0$ for $1\leq i\leq d+1$, we obtain that $\partial_{\theta_i} (\phi^k)=0$ for $1\leq i\leq d+1$. Hence $\lambda$ is an affine $k$-stress on $P$. Finally, for every $(k-1)$-face $G$ of $\partial P$,  $(-1)^{|(V\backslash M)\cap G|}\lambda_G\geq 0$, and this inequality is strict if $G\subset M$. Thus, $\lambda=\phi^k$ is a desired $k$-stress.
	\end{proof}
	
We close this section with the following remark. Kalai (private communication) speculated that for $k=2$, the following strengthening of Conjecture \ref{conj:Kalai} holds. A simplicial $d$-polytope $P$ is called {\em prime} if $\partial P$ has no missing $(d-1)$-faces.
 
\begin{conjecture} \label{conj-Kalai-2}
Let $d\geq 4$ and let $P$ be a prime simplicial $d$-polytope. Then the graph of $P$ and the space of affine $2$-stresses of $P$ uniquely determine $P$ up to affine equivalence.
\end{conjecture}

The methods of this section lead to a simple proof of this conjecture for the class of $2$-neighborly polytopes. Indeed, if $P$ is  $2$-neighborly and $\omega$ is an affine $1$-stress on $P$, then $\lambda:=\omega^2$ is an affine $2$-stress on $P$. Furthermore, for any vertex $v$ in the support of $\omega$, $\partial_{x_v}(\lambda)= 2\partial_{x_v}(\omega)\cdot \omega\in \Stress_1(P)$ is a non-zero multiple of $\omega$. We conclude that if $P$ is $2$-neighborly, then $\{\partial_{x_v}\lambda : v\in V(\partial P), \;\lambda\in\Stress_2(P)\} $ coincides with $\Stress_1(P)$, i.e., we can reconstruct $\Stress_1(P)$ from $\Stress_2(P)$. The result follows since $\Stress_1(P)$ is the space of affine dependencies of vertices of $P$, which determines $P$ up to affine equivalence.

\medskip
We hope that the tools developed and conjectures raised in this paper will lead to further progress in our understanding of affine stresses on polytopes and perhaps even to a complete resolution of Conjectures \ref{conj:Kalai}   and \ref{conj-Kalai-2}.

\section*{Acknowledgments}
We are grateful to Gil Kalai for bringing Conjecture \ref{conj:Kalai} to our attention and for his comments on the preliminary version of this paper. We also thank the referee for several helpful suggestions.

{\small
	\bibliography{refs}

\begin{thebibliography}{10}

\bibitem{AsimowRothII}
L.~Asimow and B.~Roth.
\newblock The rigidity of graphs. {II}.
\newblock {\em J. Math. Anal. Appl.}, 68(1):171--190, 1979.

\bibitem{Balinski}
M.~Balinski.
\newblock On the graph structure of convex polyhedra in $n$-space.
\newblock {\em Pacific J. Math.}, 11:431--434, 1961.

\bibitem{BjornerEdelmanZieger}
A.~Bj\"orner, P.~H. Edelman, and G.~M. Ziegler.
\newblock Hyperplane arrangements with a lattice of regions.
\newblock {\em Discrete Comput. Geom.}, 5(3):263--288, 1990.

\bibitem{Blind-Mani}
R.~Blind and P.~Mani-Levitska.
\newblock Puzzles and polytope isomorphisms.
\newblock {\em Aequationes Math.}, 34(2--3):287--297, 1987.

\bibitem{Dancis}
J.~Dancis.
\newblock Triangulated {$n$}-manifolds are determined by their
  {$[n/2]+1$}-skeletons.
\newblock {\em Topology Appl.}, 18(1):17--26, 1984.

\bibitem{Gru-book}
B.~Gr\"unbaum.
\newblock {\em Convex polytopes}, volume 221 of {\em Graduate Texts in
  Mathematics}.
\newblock Springer-Verlag, New York, second edition, 2003.
\newblock Prepared and with a preface by Volker Kaibel, Victor Klee and
  G\"unter M. Ziegler.

\bibitem{Kalai-simple}
G.~Kalai.
\newblock A simple way to tell a simple polytope from its graph.
\newblock {\em J. Combin. Theory Ser. A}, 49(2):381--383, 1988.

\bibitem{Kalai-survey}
G.~Kalai.
\newblock Some aspects of the combinatorial theory of convex polytopes.
\newblock In {\em Polytopes: abstract, convex and computational ({S}carborough,
  {ON}, 1993)}, volume 440 of {\em NATO Adv. Sci. Inst. Ser. C Math. Phys.
  Sci.}, pages 205--229. Kluwer Acad. Publ., Dordrecht, 1994.

\bibitem{Lee94}
C.~W. Lee.
\newblock Generalized stress and motions.
\newblock In {\em Polytopes: abstract, convex and computational ({S}carborough,
  {ON}, 1993)}, volume 440 of {\em NATO Adv. Sci. Inst. Ser. C Math. Phys.
  Sci.}, pages 249--271. Kluwer Acad. Publ., Dordrecht, 1994.

\bibitem{Lee96}
C.~W. Lee.
\newblock P.{L}.-spheres, convex polytopes, and stress.
\newblock {\em Discrete Comput. Geom.}, 15(4):389--421, 1996.

\bibitem{Lee-notes}
C.~W. Lee.
\newblock The $g$-theorem.
\newblock http://www.ms.uky.edu/$\sim$lee/ma715sp02/notes.pdf, 2002.

\bibitem{McMullen96}
P.~McMullen.
\newblock Weights on polytopes.
\newblock {\em Discrete Comput. Geom.}, 15(4):363--388, 1996.

\bibitem{Padrol-13}
A.~Padrol.
\newblock Many neighborly polytopes and oriented matroids.
\newblock {\em Discrete Comput. Geom.}, 50(4):865--902, 2013.

\bibitem{Pak-book}
I.~Pak.
\newblock Lecture on discrete and polyhedral geometry.
\newblock https://www.math.ucla.edu/~pak/geompol8.pdf, 2010.

\bibitem{Shemer}
I.~Shemer.
\newblock Neighborly polytopes.
\newblock {\em Israel J. Math.}, 43(4):291--314, 1982.

\bibitem{Stanley80}
R.~P. Stanley.
\newblock The number of faces of a simplicial convex polytope.
\newblock {\em Adv.~Math.}, 35:236--238, 1980.

\bibitem{Tay-et-al-I}
T.-S. Tay, N.~White, and W.~Whiteley.
\newblock Skeletal rigidity of simplicial complexes. {I}.
\newblock {\em European J. Combin.}, 16(4):381--403, 1995.

\bibitem{Tay-et-al}
T.-S. Tay, N.~White, and W.~Whiteley.
\newblock Skeletal rigidity of simplicial complexes. {II}.
\newblock {\em European J. Combin.}, 16:503--523, 1995.

\bibitem{Whiteley-84}
W.~Whiteley.
\newblock Infinitesimally rigid polyhedra. {I}. {S}tatics of frameworks.
\newblock {\em Trans. Amer. Math. Soc.}, 285(2):431--465, 1984.

\bibitem{Ziegler}
G.~M. Ziegler.
\newblock {\em Lectures on polytopes}, volume 152 of {\em Graduate Texts in
  Mathematics}.
\newblock Springer-Verlag, New York, 1995.

\end{thebibliography}
	\bibliographystyle{plain}
}
\end{document}